\documentclass[12pt]{amsart}
\usepackage[T1]{fontenc}
\usepackage{amsmath,amssymb,lmodern}
\usepackage{graphicx}
\usepackage[hidelinks]{hyperref}
\numberwithin{equation}{section}
\newtheorem{theorem}{Theorem}
\newtheorem{lemma}{Lemma}
\newtheorem{corollary}{Corollary}
\theoremstyle{remark}
\newtheorem{remark}{Remark}

\title{On the Central Maximal Function in $\mathbb{R}^n$}
\author{A. Solyanik}
\date{\today}
\dedicatory{To my dear friend Alexander Stokolos}
\address{Odessa National Polytechnic University/Odessa/Ukraine}
\email{solyanik@opu.ua}

\begin{document}
\begin{abstract}
We consider the central Hardy--Littlewood maximal operator in
$\mathbb{R}^n$. For every integer $n\geq1$, we construct two bounded,
nonnegative functions $f,g\in L^1(\mathbb{R}^n)$ that differ on a set
of positive measure and satisfy $M_cf(x)=M_cg(x)$ at every point
$x\in\mathbb{R}^n$. Thus the central maximal operator is not injective
on the nonnegative cone of $L^1(\mathbb{R}^n)$. The functions can have
compact support; when $n\geq2$, they can also be chosen radial, either
with compact support or strictly positive throughout the space.
The construction makes a small redistribution of mass invisible to
the maximal operator by excluding averaging spheres that cross its
support. We also give an explicit formula, in terms of the incomplete
beta function, for the $n$-dimensional volume of the intersection of
two balls with arbitrary radii and centers, and an elementary formula
for the central maximal function of a ball in $\mathbb{R}^3$.
\end{abstract}
\maketitle

\section{Introduction}

Let $B(x,r)$ denote the open Euclidean ball with center
$x\in\mathbb{R}^n$ and radius $r>0$. For a locally integrable function,
the central Hardy--Littlewood maximal function is
\[
M_cf(x)=\sup_{r>0}\frac{1}{|B(x,r)|}
                       \int_{B(x,r)}|f(y)|\,dy.
\]
We write $B_s=B(0,s)$ and $|B_s|=c_ns^n$, where
$c_n=\pi^{n/2}/\Gamma(n/2+1)$. All measures and integrals below are
Lebesgue measures and integrals. For background on the maximal operator
and differentiation of integrals, see de~Guzm\'an~\cite{deguzman1975}.

During the author's visit to Baylor University in September 2023,
Professor Paul Hagelstein asked whether two nonnegative functions in
$L^1(\mathbb{R}^n)$ that differ on a set of positive measure can have
the same central maximal function. We prove the following result.

\begin{theorem}
\label{main_theorem}
For each integer $n\geq1$, there are bounded, nonnegative, compactly
supported functions $f,g\in L^1(\mathbb{R}^n)$ such that
\[
\|f-g\|_1>0,
\qquad M_cf(x)=M_cg(x)\quad\text{at every }x\in\mathbb{R}^n.
\]
For $n\geq2$, the compactly supported pair can be chosen radial.
In these dimensions there are also bounded, radial functions
$f,g\in L^1(\mathbb{R}^n)$ that are strictly positive everywhere,
satisfy $\|f-g\|_1>0$, and have $M_cf(x)=M_cg(x)$ at every
$x\in\mathbb{R}^n$.
\end{theorem}

The uniqueness problem for maximal operators was studied by Lasha
Ephremidze; a survey and further references appear in
\cite{ephremidze2010}. In particular, a nonnegative function in
$L^1(\mathbb{R})$ is determined up to equality almost everywhere by
its one-sided Hardy--Littlewood maximal function
\cite[Theorem~3]{ephremidze2010}. Theorem~\ref{main_theorem} gives a
different answer for the central operator.

The idea is that the central maximal function has a ``blind spot''.
In two dimensions, we can find an annulus and, inside its hole,
a small concentric disk (the ``blind spot'') with the following
property. For every center outside the outer boundary of the annulus,
the average of its characteristic function is maximal only on disks
that either contain the small disk completely or do not intersect it
at all. This separation persists when a sufficiently small nonnegative
density is added within the small disk. It allows us to place the same
small mass at different locations within the disk without changing
the central maximal function.

The construction with a single shell no longer works in three
dimensions. There is an exterior center for which the boundary of
the unique maximizing ball passes through the origin. Such a ball
cuts through every small central ball and can distinguish between
small masses placed on opposite sides of the origin.

With both shells suitably chosen, we obtain a small central ball
(the ``blind spot'') in the hole of the inner shell. Keeping the added
density sufficiently small, we can move the same small mass freely
within this ball without changing the central maximal function.
The same construction works in all higher dimensions.

In one dimension the same mechanism is realized by two intervals of
height one, with a small bump between them. As Professor Hagelstein
communicated to the author in a private letter dated March~14, 2025,
this one-dimensional idea was proposed independently by Fedor Nazarov.

\section{The one-dimensional result}
\label{one_dimensional_result}

\begin{theorem}
\label{one_dimensional_theorem}
Let $0<\epsilon\leq1/4$ and $0<\alpha<\epsilon$. Define
$f=\chi_E+\epsilon\chi_I$, where
\[
E=(-1-\epsilon,-\epsilon)\cup(\epsilon,1+\epsilon),
\qquad I=(\alpha-\epsilon,\alpha),
\]
With $x_\epsilon=(1+3\epsilon+\epsilon^3)/(1+\epsilon^2)$, we have
\begin{equation}
\label{one_dimensional_maximal}
M_cf(x)=
\begin{cases}
\displaystyle 1-\frac{\epsilon-\epsilon^2/2}{1+\epsilon-|x|},
    & |x|\leq\epsilon,\\[8pt]
1,  & \epsilon<|x|<1+\epsilon,\\[4pt]
\displaystyle \frac{1}{2(|x|-\epsilon)},
    & 1+\epsilon\leq|x|\leq x_\epsilon,\\[8pt]
\displaystyle \frac{2+\epsilon^2}{2(|x|+1+\epsilon)},
    & |x|>x_\epsilon.
\end{cases}
\end{equation}
In particular, $M_cf$ does not depend on $\alpha$.
\end{theorem}

\begin{proof}
Write
\[
A(x,r)=\frac{1}{2r}\int_{x-r}^{x+r}f(t)\,dt.
\]
Reflection replaces $\alpha$ by $\epsilon-\alpha$, which satisfies the
same hypotheses. It therefore suffices to prove the formula for $x\geq0$,
uniformly in $\alpha$.

First, let $0\leq x\leq\epsilon$. If $0<r\leq\epsilon+x$, the
interval $(x-r,x+r)$ meets at most one tower, whose contribution has
length at most $r$. Since $f\leq\epsilon$ off the towers,
\[
A(x,r)\leq\frac{1+\epsilon}{2}
<1-\epsilon+\frac{\epsilon^2}{2}.
\]
For $\epsilon+x\leq r\leq1+\epsilon-x$, the interval contains $I$
and meets both towers without extending beyond their outer endpoints.
Hence
\[
A(x,r)=1-\frac{\epsilon-\epsilon^2/2}{r},
\]
which is increasing in $r$. For $1+\epsilon-x\leq r\leq1+\epsilon+x$,
\[
A(x,r)=\frac12+\frac{1-\epsilon+\epsilon^2-x}{2r},
\]
which is decreasing, since $1-\epsilon+\epsilon^2-x>0$.
For $r\geq1+\epsilon+x$, the average is $(2+\epsilon^2)/(2r)$ and
is again decreasing. Thus the maximum is attained at $r=1+\epsilon-x$
and equals the first expression in \eqref{one_dimensional_maximal}.

For $\epsilon<x<1+\epsilon$, we have $M_cf(x)=1$, since $f=1$ in a
neighborhood of $x$ and $0\leq f\leq1$ everywhere.

It remains to consider $x\geq1+\epsilon$. Set
\[
r_1=x-\epsilon,\qquad r_2=x+1+\epsilon.
\]
At these radii the interval contains, respectively, the entire right
tower alone and the entire support of $f$. Consequently,
\[
A(x,r_1)=\frac{1}{2(x-\epsilon)},\qquad
A(x,r_2)=\frac{2+\epsilon^2}{2(x+1+\epsilon)}.
\]
For $0<r\leq r_1$, only the right tower can be encountered and its
average is nondecreasing. For $r\geq r_2$, the total mass is fixed and
the average is decreasing. To handle the intermediate radii, put
\[
N(r)=\int_{x-r}^{x+r}f(t)\,dt,
\qquad k=\frac{1+\epsilon^2}{1+2\epsilon}.
\]
Our assumption on $\epsilon$ implies $\epsilon\leq k<1$.
For $r_1\leq r\leq x+\epsilon$, the right tower is fully contained,
the left tower is absent, and the remaining density is at most $\epsilon$.
Thus
\[
N(r)\leq1+\epsilon(r-r_1)\leq1+k(r-r_1).
\]
For $x+\epsilon\leq r\leq r_2$, we have
\[
N(r)=2+\epsilon^2-(r_2-r)
\leq2+\epsilon^2-k(r_2-r)=1+k(r-r_1).
\]
Therefore, throughout $[r_1,r_2]$,
\[
A(x,r)\leq\frac{k}{2}+\frac{1-kr_1}{2r}.
\]
The right-hand side is monotone or constant in $r$ and agrees with
$A(x,r)$ at both endpoints. It follows that
\begin{equation}
\label{one_dimensional_maximal_exterior}
M_cf(x)=\max\left\{
\frac{1}{2(x-\epsilon)},
\frac{2+\epsilon^2}{2(x+1+\epsilon)}
\right\},\qquad x\geq1+\epsilon.
\end{equation}
The two expressions are equal precisely when $x=x_\epsilon$; the first
is larger for $x<x_\epsilon$ and the second for $x>x_\epsilon$.
Since
\[
x_\epsilon-(1+\epsilon)
=\frac{\epsilon(2-\epsilon)}{1+\epsilon^2}>0,
\]
this proves the remaining cases of \eqref{one_dimensional_maximal}.
\end{proof}

For example, if $0<\alpha<\beta<\epsilon$ and
\[
g(x)=\chi_E(x)+\epsilon\chi_{(\beta-\epsilon,\beta)}(x),
\]
then $f,g\in L^1(\mathbb{R})$ are nonnegative and
\[
M_cf(x)=M_cg(x)\quad\text{for every }x\in\mathbb{R},
\qquad \|f-g\|_1=2\epsilon(\beta-\alpha)>0.
\]

\section{Proof of Theorem~\ref{main_theorem}}
\label{exclusion_section}

The one-dimensional case follows from Theorem~\ref{one_dimensional_theorem}.
Fix $n\geq2$. For a radial background $F$, its averages depend on the
center only through its distance from the origin. We therefore put
\[
A_F(x,r)=\frac{1}{c_nr^n}\int_{B(xe_1,r)}F(y)\,dy,
\qquad e_1=(1,0,\ldots,0),\quad x\geq0.
\]
The radius $r=x$ has a geometric meaning: the boundary of the averaging
ball passes through the origin. We first seek a background satisfying
\begin{equation}
\label{strict_gap}
A_F(x,x)<M_cF(xe_1)\qquad\text{for every }x>0.
\end{equation}
The next lemma turns this pointwise inequality into an interval of
excluded radii whose width is independent of the center. It also
allows the small additions needed to form our examples.

\begin{lemma}
\label{stability}
Suppose that $F$ is radial, $0\leq F\leq1$, and $F$ vanishes on
$B_\delta$ and outside $B_S$, where $0<\delta<S$. Assume that
$T=\int F>0$ and that \eqref{strict_gap} holds.
There are $0<\eta<\delta$ and constants
$\kappa_\infty,\kappa_1>0$ such that every nonnegative measurable
function $p$ with
\begin{equation}
\label{small_addition}
\|p\|_\infty\leq\kappa_\infty,
\qquad \|p\|_1\leq\kappa_1
\end{equation}
satisfies, at each $X\in\mathbb{R}^n$,
\begin{equation}
\label{safe_radii}
M_c(F+p)(X)=
\sup_{\substack{r>0\\|r-\|X\||>\eta}}
\frac{1}{c_nr^n}\int_{B(X,r)}(F+p).
\end{equation}
The constants depend on the fixed background and the dimension.
The addition $p$ need not be radial or compactly supported.
\end{lemma}

\begin{proof}
We compare the excluded averages with larger ones in three ranges
of centers. By radial symmetry we may write the center as $xe_1$;
rotating the addition $p$ does not change either of its norms.
Put $a=\delta/4$ and initially restrict $0<\eta<\min\{a,1\}$.

For $0\leq x\leq a$ and $|r-x|\leq\eta$, the averaging ball lies
in $B_\delta$. Its average of $F+p$ is at most $\|p\|_\infty$.
The ball of radius $x+S$ contains the support of $F$ and has average
at least $T/[c_n(S+a)^n]$. Choosing $\kappa_\infty$ smaller than
this positive number excludes the indicated radii, including at $x=0$.

For distant centers, radial symmetry places half of the mass $T$
in the half-space $y_1>0$. If $r=x+t$ with $|t|\leq\eta$ and
$x\geq2$, then every $y\in B(xe_1,r)$ satisfies
\[
y_1>\frac{\|y\|^2-2xt-t^2}{2x}\geq-2\eta.
\]
The extra slab has thickness at most $2\eta$. Since $0\leq F\leq1$
and its support lies in $B_S$, its mass in this slab is at most
$C\eta$, where $C$ depends only on $S$ and $n$. Take $\eta$ and
$\kappa_1$ small enough that
\[
\int_{B(xe_1,r)}(F+p)\leq T/2+C\eta+\kappa_1\leq3T/4.
\]
We can fix $R>\max\{2,S,a\}$ sufficiently large that, for all $x\geq R$,
\[
\frac{3T}{4c_n(x-1)^n}<\frac{T}{c_n(x+S)^n}.
\]
Since $r\geq x-1$, every excluded average is smaller than the average
at radius $x+S$, which is allowed because $S>\eta$.

It remains to consider $a\leq x\leq R$. The averages of $F$ are
jointly continuous in the center and positive radius, whereas $M_cF$
is lower semicontinuous. By \eqref{strict_gap} and compactness,
\[
\gamma=\min_{a\leq x\leq R}
       \bigl(M_cF(xe_1)-A_F(x,x)\bigr)>0.
\]
Uniform continuity allows us to decrease $\eta$ so that
$A_F(x,r)\leq M_cF(xe_1)-3\gamma/4$ whenever $|r-x|\leq\eta$
in this range. Decrease $\kappa_\infty$ also so that
$\kappa_\infty\leq\gamma/4$. The excluded averages of $F+p$ are
then at most $M_cF(xe_1)-\gamma/2$. For $F$ itself the supremum
can already be taken over the remaining radii, and $p\geq0$ cannot
decrease those averages. This proves \eqref{safe_radii} on the
compact range. The reductions of $\eta$ and $\kappa_\infty$ preserve
the preceding comparisons, completing the proof.
\end{proof}

The conclusion has the desired geometric consequence. If two additions
$p,q$ satisfy \eqref{small_addition}, agree outside $B_\eta$, and
have the same integral over $B_\eta$, then
\begin{equation}
\label{redistribution}
M_c(F+p)(X)=M_c(F+q)(X)\quad\text{at every }X\in\mathbb{R}^n.
\end{equation}
In fact, an allowed ball either misses $B_\eta$ or contains it
completely. Its integrals of $p$ and $q$ are equal in either case.
Taking the suprema in \eqref{safe_radii} proves the assertion.
The remainder of the construction is therefore devoted to
\eqref{strict_gap}.

We now construct a background with this property.

\begin{lemma}
\label{background}
For each fixed integer $n\geq2$, there are numbers
$0<\delta<1/3$, $\tau>0$, and $\lambda>1+\tau$ such that
\begin{equation}
\label{two_shell_background}
F(y)=\chi_{\{\delta<\|y\|<1\}}(y)
     +\chi_{\{\lambda-\tau<\|y\|<\lambda+\tau\}}(y)
\end{equation}
satisfies \eqref{strict_gap}.
\end{lemma}

Thus the background consists of two concentric shells. The small
region in which mass will be redistributed lies in the central hole,
inside the inner shell. It is not in the gap between the two shells.
The proof first chooses the outer shell and then chooses the hole.
Figure~\ref{fig:background} shows this order.

\begin{figure}[htbp]
\centering
\includegraphics[width=429.39pt,trim=3pt 3pt 3pt 3pt,clip]{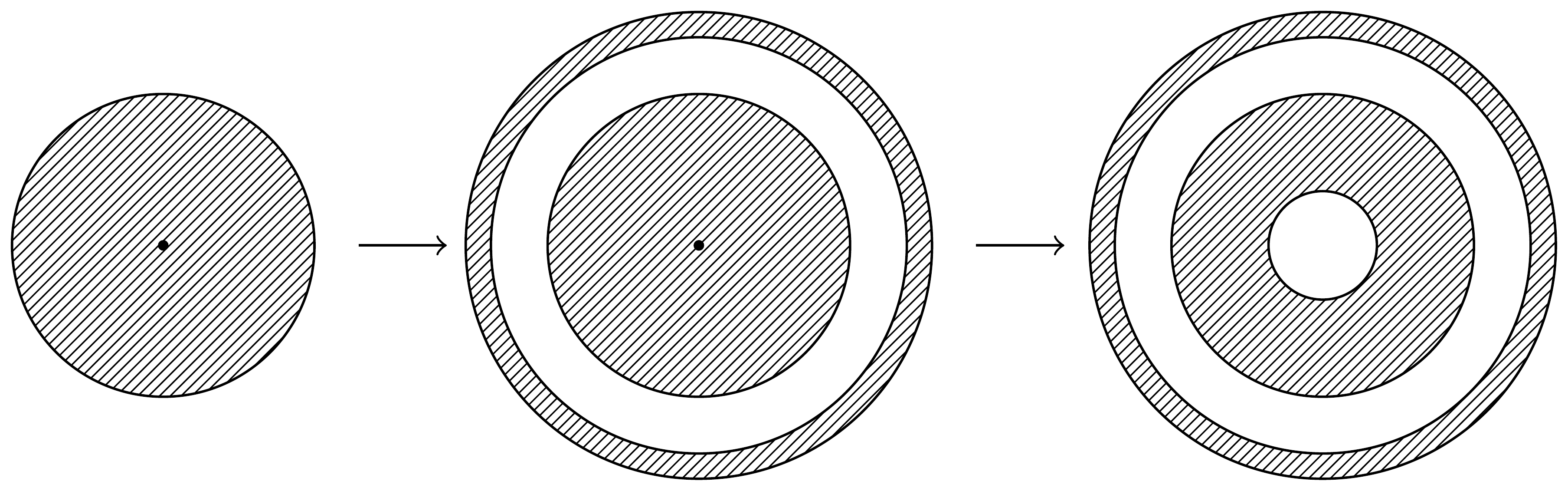}
\caption{The hatched regions are those where the background equals one.
The central hole and the gap between the two shells remain empty.}
\label{fig:background}
\end{figure}

\subsection{The whole ball has one exterior stationary center}

We only need to differentiate an average at $r=x$; there is no need
to find all its maximizing radii. Write
\[
I_W(x,r)=\int_{B(xe_1,r)}W(y)\,dy,
\qquad \kappa_n=(n-1)c_{n-1},
\]
and, whenever the radial derivative exists, put
\[
D_W(x)=\frac{x\partial_r I_W(x,x)-nI_W(x,x)}{\kappa_n}.
\]
This normalization gives
\begin{equation}
\label{diagonal_derivative}
\partial_r A_W(x,x)=\frac{\kappa_n}{c_nx^{n+1}}D_W(x).
\end{equation}
In particular, $D_W(x)\ne0$ rules out $r=x$ as a maximizing radius.

For the indicator of a ball, this derivative can be described by one
function. Define
\[
Q_n(a)=\int_a^1(1-t^2)^{(n-3)/2}\,dt\qquad(0\leq a\leq1)
\]
and
\begin{equation}
\label{phi}
\Phi_n(t)=
\begin{cases}
0,&0\leq t\leq1/2,\\[3pt]
\displaystyle\frac{t}{n-1}
 \left(1-\frac{1}{4t^2}\right)^{(n-1)/2}
 -Q_n\!\left(\frac{1}{2t}\right),&t>1/2.
\end{cases}
\end{equation}
The integral defining $Q_n$ is finite also when $n=2$.

\begin{lemma}
\label{ball_derivative}
For $s>0$ and $x>0$, the derivative in
\eqref{diagonal_derivative} exists for $W=\chi_{B_s}$ and
\begin{equation}
\label{ball_kernel}
D_{\chi_{B_s}}(x)=s^n\Phi_n(x/s).
\end{equation}
On $[1,\infty)$, $\Phi_n$ has exactly one zero $a_n$.
This zero satisfies $a_n>1$ and $\Phi_n'(a_n)>0$.
\end{lemma}

\begin{proof}
When $2x<s$, the averaging ball lies inside $B_s$ for radii near
$x$, so its average is one and both sides of \eqref{ball_kernel}
vanish. At $2x=s$, the volume lost at internal tangency is of order
$|r-x|^{(n+1)/2}$, which is higher than first order; the derivative
is therefore zero there as well.

Suppose that $2x>s$ and put $W=\chi_{B_s}$. The two spheres meet
transversely, so the intersection volume is differentiable in $r$
near $x$. Apply the divergence theorem to the field $y-xe_1$ on
$B(xe_1,x)\cap B_s$. Its flux through the averaging sphere is
$x\partial_r I_W(x,x)$, while its outward normal component on
$\partial B_s$ is $s-xy_1/s$. Consequently,
\[
x\partial_r I_W(x,x)-nI_W(x,x)
=\int_{\partial B_s\cap B(xe_1,x)}(xy_1/s-s)\,dS(y).
\]
The cap on the right is given by $y_1/s>s/(2x)$. Using $t=y_1/s$
as the angular variable and dividing by $\kappa_n$, we obtain
\[
D_W(x)=s^{n-1}\int_{s/(2x)}^1
           (xt-s)(1-t^2)^{(n-3)/2}\,dt
       =s^n\Phi_n(x/s).
\]
The last equality follows by integrating the term containing $t$
and using the definition of $Q_n$.

To locate the zero, differentiate \eqref{phi} for $t>1/2$:
\begin{equation}
\label{phi_prime}
\Phi_n'(t)=\frac{1}{n-1}
 \left(1-\frac{1}{4t^2}\right)^{(n-3)/2}
 \left(1-\frac{n}{4t^2}\right).
\end{equation}
Also, integration of $t(1-t^2)^{(n-3)/2}$ gives
\[
\Phi_n(1)=-\int_{1/2}^1(1-t)(1-t^2)^{(n-3)/2}\,dt<0.
\]
At infinity we have
\begin{equation}
\label{phi_tail}
\Phi_n(t)=\frac{t}{n-1}-Q_n(0)+O(t^{-1}).
\end{equation}
By \eqref{phi_prime}, the function decreases on $(1/2,\sqrt n/2)$
and increases thereafter. On $[1,\infty)$ it therefore either
increases throughout or first decreases and then increases.
Its negative value at one and its positive linear growth give
exactly one zero, on the interval of strict increase.
\end{proof}

\subsection{Placing the stationary center inside a thin shell}

We now add an outer shell while keeping the inner ball whole. For a
small thickness $\tau>0$, consider
\[
H_{\tau,\lambda}
=\chi_{B_1}+\chi_{\{\lambda-\tau<\|y\|<\lambda+\tau\}},
\qquad \lambda-\tau>1.
\]
The shell slightly shifts the stationary center of the ball.
Its middle radius must be chosen together with this shift.

Set $K(s,x)=s^n\Phi_n(x/s)$. By linearity and
\eqref{ball_kernel},
\begin{equation}
\label{shell_derivative}
D_{H_{\tau,\lambda}}(x)=\Phi_n(x)+R_{\tau,\lambda}(x),
\qquad R_{\tau,\lambda}(x)=K(\lambda+\tau,x)-K(\lambda-\tau,x).
\end{equation}
Choose a small closed interval $J\subset(1,\infty)$ containing
$a_n$ in its interior, with $\Phi_n'>0$ on $J$ and
$\sup J<2\inf J$. Choose a smaller closed interval $J_0$ in the
interior of $J$, still containing $a_n$ in its interior. We allow $\lambda$
to vary in $J_0$ and take $\tau>0$ so small that
$\lambda-\tau,\lambda+\tau\in J$ throughout this range.

We first estimate the perturbation uniformly in $\lambda$. For
$x,s\in J$, the ratio $x/s$ stays strictly above $1/2$, where $K$
is smooth. The mean value theorem therefore gives
\begin{equation}
\label{local_shell_bound}
\sup_{x\in J}\bigl(|R_{\tau,\lambda}(x)|
                  +|\partial_xR_{\tau,\lambda}(x)|\bigr)
\leq C\tau,
\end{equation}
with $C$ independent of $\lambda\in J_0$ and small $\tau$.
For the global estimate we use
\[
\partial_sK(s,x)=s^{n-1}
       \bigl(n\Phi_n(x/s)-(x/s)\Phi_n'(x/s)\bigr).
\]
When $n\geq4$, $\Phi_n$ is continuously differentiable also at
$1/2$, with bounded derivative, and $\Phi_n(t)=O(1+t)$.
In dimensions two and three, we keep the arguments away from this
point: $a_2<2$, since
$\Phi_2(2)=\sqrt{15}/2-\arccos(1/4)>0$, and $a_3=3/2$, since
$\Phi_3(t)=t/2-1+3/(8t)$ for $t>1/2$.
In these two dimensions choose $J$ (and then $J_0$) also with
$\sup J<2$.
Then $x/s\geq1/\sup J>1/2$ for $x\geq1$, $s\in J$,
and the same derivative bounds hold on the required range.
We obtain in every dimension
\begin{equation}
\label{global_shell_bound}
|R_{\tau,\lambda}(x)|\leq C\tau(1+x)
\qquad\text{for }x\geq1.
\end{equation}

The unique zero and the linear asymptotic \eqref{phi_tail} give
$|\Phi_n(x)|\geq c(1+x)$ for $x\geq1$ outside $J$, with $c>0$.
Choose $\tau$ so small that \eqref{global_shell_bound} preserves
this sign and \eqref{local_shell_bound} makes
$D_{H_{\tau,\lambda}}$ strictly increasing on $J$, uniformly for
$\lambda\in J_0$.

At the endpoints of $J_0$, the function $\Phi_n$ has opposite signs.
For sufficiently small $\tau$, the uniform bound
\eqref{local_shell_bound} preserves these signs for the continuous
function $\lambda\mapsto\Phi_n(\lambda)+R_{\tau,\lambda}(\lambda)$.
The intermediate value theorem therefore gives a $\lambda\in J_0$
such that
\begin{equation}
\label{matching}
\Phi_n(\lambda)+R_{\tau,\lambda}(\lambda)=0.
\end{equation}
For this choice, $D_{H_{\tau,\lambda}}$ is nonzero outside $J$ and
strictly increasing on $J$. Its only zero on $[1,\infty)$ is
therefore $x=\lambda$, which lies in the outer shell.

Fix this shell, and write $H=H_{\tau,\lambda}$.
The exceptional center $\lambda e_1$ now lies inside a region where
$H=1$. Small balls centered there have average one. The ball of
radius $\lambda$ meets the gap between the unit ball and the outer
shell in positive measure, so its average is strictly less than one.
Thus the stationary center has been dealt with by geometry.

\subsection{Removing the central ball last}

We keep $\tau$ and $\lambda$ fixed and remove $B_\delta$.
Let
\[
N=(\lambda-\tau/2,\lambda+\tau/2).
\]
This neighborhood lies inside the outer shell. Since $D_H$ has no
zero outside $N$, and \eqref{global_shell_bound} preserves its
linear growth at infinity, there is a constant $b>0$ such that
\[
|D_H(x)|\geq b(1+x)\qquad(x\geq1,\ x\notin N).
\]
The constant is chosen after the shell has been fixed.
For the background $F$ in \eqref{two_shell_background},
\[
D_F(x)=D_H(x)-\delta^n\Phi_n(x/\delta).
\]
Since $|\Phi_n(t)|\leq C_n(1+t)$, the change for $x\geq1$
is bounded by
\begin{equation}
\label{hole_error}
\delta^n|\Phi_n(x/\delta)|
\leq C_n\delta^{n-1}(1+x)\qquad(0<\delta<1).
\end{equation}
Here $n\geq2$ is essential: the coefficient tends to zero with
$\delta$. Choose $0<\delta<1/3$ so that this coefficient is
less than $b/2$. Then $D_F(x)\ne0$ for $x\geq1$ outside $N$.
At these centers \eqref{diagonal_derivative} proves the strict gap.
For $x\in N$, small balls still have average one, whereas the
ball of radius $x$ meets the central hole and has smaller average.
We do not need to locate the zeros that remain inside $N$.

The interior centers require only volume comparisons. For
$\delta<x<1$, the center lies inside the inner shell, and the
same comparison with average one applies. If $0<x\leq\delta$,
then
\[
B\left(\frac x2e_1,\frac x2\right)
\subset B(xe_1,x)\cap B_\delta.
\]
The diagonal ball thus loses at least a fraction $2^{-n}$ of its
volume, and $A_F(x,x)\leq1-2^{-n}$. In contrast,
$B(xe_1,1-x)$ lies inside $B_1$ and contains the whole hole,
because $1-2x\geq1-2\delta>\delta$. Its average is
\[
A_F(x,1-x)=1-\left(\frac{\delta}{1-x}\right)^n>1-2^{-n},
\]
where we used $\delta/(1-x)\leq\delta/(1-\delta)<1/2$.
This proves \eqref{strict_gap} at every $x>0$, and completes
the proof of Lemma~\ref{background}.

\subsection{Redistributing the small mass}

\begin{proof}[Proof of Theorem~\ref{main_theorem} for $n\geq2$]
Take $F$ from Lemma~\ref{background} and then
$\eta,\kappa_\infty,\kappa_1$ from Lemma~\ref{stability}.
The annulus $B_\eta\setminus B_{\eta/2}$ has $2^n-1$ times
the volume of $B_{\eta/2}$. Hence the two additions
\[
p=\epsilon\chi_{B_{\eta/2}},
\qquad q=\frac{\epsilon}{2^n-1}
                      \chi_{B_\eta\setminus B_{\eta/2}}
\]
have the same mass $\epsilon c_n(\eta/2)^n$. Choose $\epsilon>0$
small enough that both satisfy \eqref{small_addition}.
Equation~\eqref{redistribution} gives identical maximal functions
for $f=F+p$ and $g=F+q$ at every point. These functions are radial,
bounded, nonnegative, and compactly supported, and
\[
\|f-g\|_1=2\epsilon c_n(\eta/2)^n>0.
\]

To obtain a strictly positive pair, add the same positive tail
$\epsilon e^{-\|y\|}$ to both functions. Decrease $\epsilon$ if
necessary so that $p+\epsilon e^{-\|y\|}$ and
$q+\epsilon e^{-\|y\|}$ still satisfy \eqref{small_addition};
this is possible because all these terms scale with $\epsilon$
and the exponential belongs to $L^1\cap L^\infty$. The two additions
agree outside $B_\eta$ and have equal integrals inside it. Thus
\eqref{redistribution} applies again, giving identical maximal
functions for
\[
f(y)=F(y)+p(y)+\epsilon e^{-\|y\|},
\qquad g(y)=F(y)+q(y)+\epsilon e^{-\|y\|}.
\]
These functions are bounded, radial, integrable, and strictly positive
everywhere. Their difference is still $p-q$, so the same positive
$L^1$ distance is preserved.
\end{proof}

One may instead translate a small ball inside $B_\eta$.
For fixed $0<\sigma<\eta$ and sufficiently small $\epsilon>0$,
\[
f_a=F+\epsilon\chi_{B(a,\sigma)},\qquad \|a\|+\sigma<\eta,
\]
has a maximal function independent of $a$. All these additions have
mass $\epsilon c_n\sigma^n$. For example, taking $\sigma=\eta/8$
and $a=\pm(\eta/4)e_1$ gives disjoint added balls, and hence two
different functions. This is the geometric version of the same
construction: the small ball moves in the central cavity enclosed
by both shells.

\section{An example in \texorpdfstring{$\mathbb{R}^2$}{R2}}
\label{planar_annulus}

In the plane, one annulus suffices. Put $\delta=1/2$ and
\[
E_\delta=\{y\in\mathbb{R}^2:\delta<\|y\|<1\}.
\]

\begin{theorem}
\label{two_dimensional_theorem}
There are constants $0<\eta<1/2$ and $\epsilon_0>0$ such that,
with $\sigma=\eta/8$ and $a=(\eta/4)e_1$, the functions
\[
f=\chi_{E_\delta}+\epsilon\chi_{B(a,\sigma)},
\qquad
g=\chi_{E_\delta}+\epsilon\chi_{B(-a,\sigma)}
\]
satisfy, for each $0<\epsilon\leq\epsilon_0$,
\[
M_cf(X)=M_cg(X)\quad\text{at every }X\in\mathbb{R}^2,
\qquad \|f-g\|_1=2\epsilon\pi\sigma^2>0.
\]
Both functions are bounded, nonnegative, and compactly supported.
\end{theorem}

\begin{proof}
Write $A(x,r)=|E_\delta\cap B(xe_1,r)|/(\pi r^2)$.
We first prove
\begin{equation}
\label{two_dimensional_strict_gap}
A(x,x)<M_c(\chi_{E_\delta})(xe_1)\qquad(x>0),
\end{equation}
and then apply Lemma~\ref{stability}.
For $s>0$, put $H_s(x)=2s^2\Phi_2(x/s)$.
Formula~\eqref{phi} gives $H_s(x)=0$ when $2x\leq s$, and
\begin{equation}
\label{two_dimensional_Hs}
H_s(x)=s\sqrt{4x^2-s^2}
       -2s^2\arccos\!\left(\frac{s}{2x}\right)
       \qquad(2x>s).
\end{equation}
By Lemma~\ref{ball_derivative},
\begin{equation}
\label{two_dimensional_diagonal_derivative}
\pi x^3\partial_r A(x,x)=H_1(x)-H_\delta(x).
\end{equation}
We use this derivative where its sign is definite, and compare with
a ball containing the annulus when the center is farther away.

If $0<x\leq1/4$, then $B((x,0),x)$ lies in the hole, and
$A(x,x)=0<M_c(\chi_{E_\delta})(xe_1)$.

If $1/4<x\leq1/2$, then $H_1(x)=0$. Set
$\theta=\arccos(\delta/(2x))$, so that $0<\theta\leq\pi/3$.
Formula \eqref{two_dimensional_Hs} becomes
\[
H_\delta(x)=\delta^2(\tan\theta-2\theta)<0.
\]
Indeed, $\cos\theta\geq1/2$ and $\sin\theta<\theta$ imply
$\tan\theta\leq2\sin\theta<2\theta$.
Thus $\partial A/\partial r(x,x)>0$, which proves
\eqref{two_dimensional_strict_gap} in this case.

If $1/2<x<1$, the center lies in the annulus, and
$M_c(\chi_{E_\delta})(xe_1)=1$. The disk $B((x,0),x)$ intersects the
hole in positive area, so $A(x,x)<1$.

Next, suppose that $1\leq x\leq2$, and put
$D(x)=H_1(x)-H_\delta(x)$. Define
\[
q(t)=\frac{1-2t}{\sqrt{1-t}}.
\]
Differentiating \eqref{two_dimensional_Hs}, and using $\delta=1/2$,
we obtain
\[
D'(x)=2q\!\left(\frac{1}{4x^2}\right)
          -q\!\left(\frac{1}{16x^2}\right).
\]
The function $q$ is decreasing on $[0,1/4]$, with $q(0)=1$ and
$q(1/4)=1/\sqrt3$. Hence
\[
D'(x)\geq\frac{2}{\sqrt3}-1>0\qquad(x\geq1).
\]
On the other hand,
\begin{align*}
D(2)
&=\sqrt{15}-\frac{\sqrt{63}}4-\frac{3\pi}{4}
  +2\arcsin\!\left(\frac14\right)
  -\frac12\arcsin\!\left(\frac18\right)\\
&\leq\frac{17}{\sqrt{15}}-\frac{\sqrt{63}}4
      -\frac{3\pi}{4}-\frac1{16}<0.
\end{align*}
Here we used $t\leq\arcsin t\leq t/\sqrt{1-t^2}$.
For completeness, the last strict inequality follows entirely from
rational bounds: $\sqrt{15}>58/15$, $\sqrt{63}>198/25$, and
$\pi>157/50$ give
\[
\frac{17}{\sqrt{15}}-\frac{\sqrt{63}}4-\frac{3\pi}{4}-\frac1{16}
<\frac{255}{58}-\frac{99}{50}-\frac{471}{200}-\frac1{16}
=-\frac{11}{11600}.
\]
Thus $D(x)\leq D(2)<0$ for $1\leq x\leq2$. By
\eqref{two_dimensional_diagonal_derivative}, decreasing the radius
slightly increases the average, and \eqref{two_dimensional_strict_gap}
follows.

Finally, let $x\geq2$. Integration in polar coordinates gives
\begin{align*}
\pi x^2 A(x,x)
&=2\int_{1/2}^1 s\arccos\!\left(\frac{s}{2x}\right)\,ds\\
&=\frac{3\pi}{8}
  -2\int_{1/2}^1 s\arcsin\!\left(\frac{s}{2x}\right)\,ds\\
&\leq\frac{3\pi}{8}-\frac{7}{24x}.
\end{align*}
The disk of radius $x+1$ contains the whole annulus. Since $\pi<7/2$,
the preceding bound gives
\[
A(x,x)<\frac{3}{8x^2}-\frac{1}{12x^3}
\leq\frac{3}{4(x+1)^2}=A(x,x+1).
\]
The middle inequality follows from
\[
\frac{3}{4(x+1)^2}
 -\left(\frac{3}{8x^2}-\frac{1}{12x^3}\right)
=\frac{(x-2)(9x^2+2x-1)}{24x^3(x+1)^2}\geq0.
\]
This proves \eqref{two_dimensional_strict_gap}.
Apply Lemma~\ref{stability} to $F=\chi_{E_\delta}$ and take the
resulting $\eta,\kappa_\infty,\kappa_1$. Set $\sigma=\eta/8$,
$a=(\eta/4)e_1$, and
\[
\epsilon_0=\min\left\{\kappa_\infty,
                           \frac{\kappa_1}{\pi\sigma^2}\right\}.
\]
The two added disks are disjoint and lie inside $B_\eta$.
Their indicators, multiplied by $\epsilon\leq\epsilon_0$, satisfy
\eqref{small_addition} and have equal mass $\epsilon\pi\sigma^2$.
Equation~\eqref{redistribution} proves the equality of the maximal
functions at every center. Disjointness gives the stated $L^1$ distance.
\end{proof}

\section{The volume of the intersection of two balls}
\label{volume_section}

In this section, we derive a formula for the volume of the intersection
of two $n$-dimensional balls of arbitrary radii. This formula, expressed
in terms of the beta function, is included here solely because we have
not found it in the literature available to us, either in reference
books or in research articles.
Throughout this section, $n\geq2$.

\subsection{The volume of a cap with signed distance}

The incomplete beta function is defined by
\begin{equation}
\label{incomplete_beta}
\mathrm{B}(a,b;z)=\int_0^z t^{a-1}(1-t)^{b-1}\,dt,
\qquad a,b>0,\quad 0\leq z\leq1.
\end{equation}
The complete beta function is obtained when $z=1$ and can be expressed
in terms of the gamma function as
\begin{equation}
\label{complete_beta}
\mathrm{B}(a,b)=\mathrm{B}(a,b;1)
=\frac{\Gamma(a)\Gamma(b)}{\Gamma(a+b)},
\end{equation}
where
\[
\Gamma(z)=\int_0^\infty t^{z-1}e^{-t}\,dt,\qquad z>0.
\]

The volume of the $n$-dimensional ball $B_r=B(x,r)$ is given by
\begin{equation}
\label{ball_volume}
|B_r|=c_n r^n,
\end{equation}
where, here and throughout the paper, $|E|$ denotes the $n$-dimensional
Lebesgue measure of a measurable set $E$, and
\begin{equation}
c_n=\frac{\pi^{n/2}}{\Gamma\!\left(\frac n2+1\right)}.
\end{equation}

The function $\mathrm{B}(a,b;z)$ is not regularized: its value at
$z=1$ is $\mathrm{B}(a,b)$, rather than one.

Let $H$ be an open half-space in $\mathbb{R}^n$, written as
\[
H=z+\{y\in\mathbb{R}^n:\langle y,e\rangle>0\},
\qquad z\in\mathbb{R}^n,\quad \|e\|=1.
\]
Here $\langle x,y\rangle=\sum_{j=1}^n x_jy_j$ is the scalar product.
The boundary
\[
\partial H=z+\{y\in\mathbb{R}^n:\langle y,e\rangle=0\}
\]
is an $(n-1)$-dimensional hyperplane.

For a ball $B=B(x,r)$, its intersection $C_h=B\cap H$ is called a
\emph{hyperspherical cap}. We denote by $h$ the \emph{signed distance}
from the center $x$ to the cutting hyperplane $\partial H$, with the
following sign convention:
\[
h=
\begin{cases}
\operatorname{dist}(x,\partial H),&x\notin H,\\
-\operatorname{dist}(x,\partial H),&x\in H.
\end{cases}
\]
Thus $h<0$ when the center lies in the half-space defining the cap,
and $h=0$ when the cutting hyperplane passes through the center.
For $-r\leq h\leq r$, the usual geometric height of the cap is $r-h$.

By translation and rotation invariance, the volume of a cap depends only
on $r$ and $h$. We may therefore take $B_r=B(0,r)$ and
$H=\{x\in\mathbb{R}^n:x_n>h\}$.

\begin{lemma}
Let $r>0$, $-r\leq h\leq r$, and
\[
C_r^h=\{x=(x_1,\ldots,x_n)\in B_r:x_n>h\},
\qquad B_r=B(0,r).
\]
Then
\begin{equation}
\label{cap_volume}
|C_r^h|=c_{n-1}(2r)^n\,
\mathrm{B}\!\left(
\frac{n+1}{2},\frac{n+1}{2};
\frac12\left(1-\frac hr\right)\right).
\end{equation}
In particular,
\[
|C_r^r|=0,\qquad |C_r^{-r}|=c_n r^n,
\qquad |C_r^0|=\frac12c_n r^n.
\]
\end{lemma}

\begin{proof}
The section of $B_r$ by the hyperplane $x_n=t$, where $|t|<r$,
is an $(n-1)$-dimensional ball of radius $\sqrt{r^2-t^2}$.
Integrating these sections and then making the substitution
$t=r(1-2v)$, we obtain
\[
\begin{aligned}
|C_r^h|
&=c_{n-1}\int_h^r(r^2-t^2)^{(n-1)/2}\,dt\\
&=c_{n-1}(2r)^n
  \int_0^{(1-h/r)/2}v^{(n-1)/2}(1-v)^{(n-1)/2}\,dv.
\end{aligned}
\]
The last integral is the incomplete beta function in
\eqref{cap_volume}. The three particular values also follow directly
from the geometry of the corresponding caps.
\end{proof}

For instance, when $n=3$, evaluating $\mathrm{B}(2,2;z)$ gives
\begin{equation}
\label{cap_volume_three}
\begin{aligned}
\mathrm{Vol}(C_r^h)
&=8\pi r^3\mathrm{B}\!\left(2,2;\frac12\left(1-\frac hr\right)\right)\\
&=\frac{2\pi r^3}{3}-\pi r^2h+\frac{\pi h^3}{3}
 =\frac{\pi}{3}(r-h)^2(2r+h).
\end{aligned}
\end{equation}

\subsection{The volume of the intersection of two balls}

Consider two balls $B_1=B(x_1,r_1)$ and $B_2=B(x_2,r_2)$ in
$\mathbb{R}^n$, with $r_1,r_2>0$ and $d=\|x_1-x_2\|$.
If $d\leq|r_1-r_2|$, one ball is contained in the other. If
$d\geq r_1+r_2$, the open balls are disjoint. Thus
\[
|B_1\cap B_2|=
\begin{cases}
c_n\min\{r_1,r_2\}^n,&d\leq|r_1-r_2|,\\[2pt]
0,&d\geq r_1+r_2.
\end{cases}
\]
The first case includes coincident centers, $d=0$. It remains to treat
partial overlap, for which the following formula applies.

\begin{theorem}
Let $B_1$ and $B_2$ be two balls in $\mathbb{R}^n$ with positive radii
$r_1$ and $r_2$ and centers at distance $d$ apart, where
\begin{equation}
\label{condition}
|r_1-r_2|<d<r_1+r_2.
\end{equation}
Then
\begin{equation}
\label{vol_inter_two_balls}
\begin{aligned}
|B_1\cap B_2|
={}&\frac{2^n\pi^{(n-1)/2}}{\Gamma((n+1)/2)}
\Bigg[
 r_1^n\,\mathrm{B}\!\left(
 \frac{n+1}{2},\frac{n+1}{2};
 \frac12\left(1-\frac{h_1}{r_1}\right)\right)\\
&\hspace{55pt}{}+
 r_2^n\,\mathrm{B}\!\left(
 \frac{n+1}{2},\frac{n+1}{2};
 \frac12\left(1-\frac{h_2}{r_2}\right)\right)
\Bigg],
\end{aligned}
\end{equation}
where
\begin{align}
h_1&=\frac{r_1^2-r_2^2+d^2}{2d},\label{h_mu}\\
h_2&=\frac{r_2^2-r_1^2+d^2}{2d}.\label{h_r}
\end{align}
These signed distances satisfy $-r_j<h_j<r_j$ for $j=1,2$, so both
arguments of the incomplete beta function lie in $(0,1)$.
\end{theorem}

\begin{proof}
Put $e=(x_2-x_1)/d$ and $t=\langle y-x_1,e\rangle$. Subtracting
the two sphere equations gives
\[
\bigl(\|y-x_2\|^2-r_2^2\bigr)
 -\bigl(\|y-x_1\|^2-r_1^2\bigr)
=2d(h_1-t),
\]
where $h_1$ is given by \eqref{h_mu}. Thus the common boundary lies
in the hyperplane $L=\{t=h_1\}$. The signed distances for the two
caps are $h_1$ and $h_2=d-h_1$, giving \eqref{h_r}. Under
\eqref{condition}, the common section has positive radius, so
\[
r_1^2-h_1^2=r_2^2-h_2^2>0.
\]

The same subtraction explains the decomposition. On the side $t>h_1$,
every point of $B_1$ also belongs to $B_2$; on the side $t<h_1$,
every point of $B_2$ also belongs to $B_1$. Apart from its section
in $L$, which has zero volume, the overlap is therefore the disjoint
union of these two caps. Applying \eqref{cap_volume} to each cap
and adding their volumes proves \eqref{vol_inter_two_balls}.
\end{proof}

\begin{figure}[htbp]
\centering
\includegraphics[width=271.77pt,trim=3pt 3pt 3pt 3pt,clip]{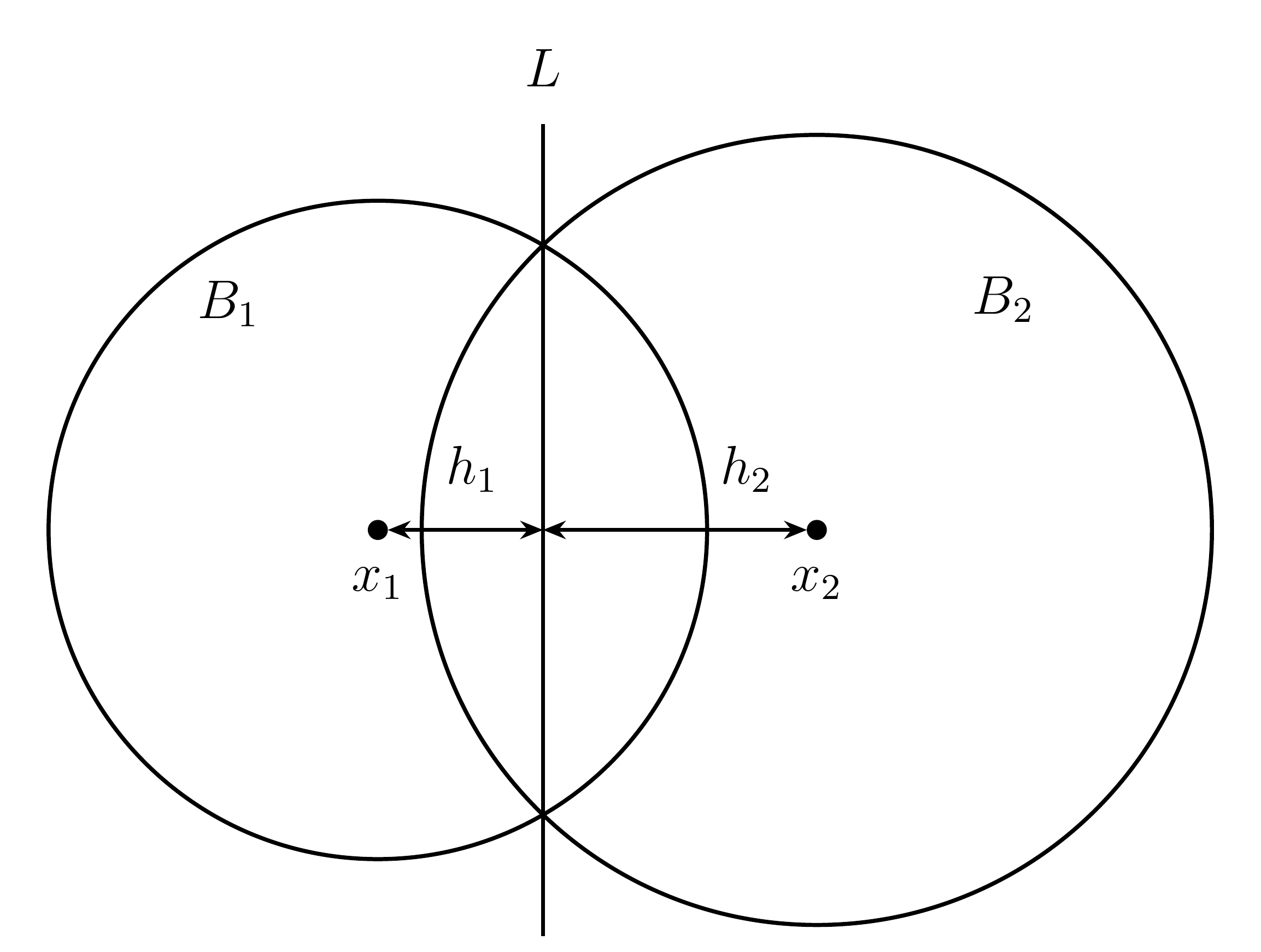}
\caption{The common chord lies in $L$. The overlap is the cap of
$B_1$ on the right of $L$ and the cap of $B_2$ on the left.}
\label{fig:two_balls_intersection}
\end{figure}

\begin{corollary}[$n=2$]
For two disks $D_1$ and $D_2$ with radii $r_1,r_2$ and centers at
distance $d$ apart, where $|r_1-r_2|<d<r_1+r_2$, we have
\begin{equation}
\begin{aligned}
\text{Area}(D_1\cap D_2)
={}&r_1^2\arccos\!\left(\frac{d^2+r_1^2-r_2^2}{2dr_1}\right)\\
&{}+r_2^2\arccos\!\left(\frac{d^2+r_2^2-r_1^2}{2dr_2}\right)\\
&{}-\frac12\sqrt{(r_1+r_2-d)(r_1-r_2+d)
                 (r_2-r_1+d)(r_1+r_2+d)}.
\end{aligned}
\end{equation}
\end{corollary}

\begin{proof}
In dimension two, integration of the sections of a cap gives
\[
|C_r^h|=2\int_h^r\sqrt{r^2-t^2}\,dt
=r^2\arccos(h/r)-h\sqrt{r^2-h^2}.
\]
Add this expression for the two caps and use
$h_1+h_2=d$ and $r_1^2-h_1^2=r_2^2-h_2^2$.
The square-root contribution is $-d\sqrt{r_1^2-h_1^2}$,
and the claimed expression follows from the identity
\begin{align*}
4d^2(r_1^2-h_1^2)
={}&(r_1+r_2-d)(r_1-r_2+d)\\
&\qquad{}\times(r_2-r_1+d)(r_1+r_2+d).\qedhere
\end{align*}
\end{proof}

\begin{corollary}[$n=3$]
For two balls $B_1$ and $B_2$ with radii $r_1,r_2$ and centers at
distance $d$ apart, where $|r_1-r_2|<d<r_1+r_2$, we have
\begin{equation}
\label{vol_3_balls_inter}
\begin{aligned}
\text{Vol}(B_1\cap B_2)
={}&\frac{\pi}{12d}(r_1+r_2-d)^2\\
&\quad{}\times\bigl(d^2+2d(r_1+r_2)-3(r_1-r_2)^2\bigr).
\end{aligned}
\end{equation}
\end{corollary}

\begin{proof}
By \eqref{cap_volume_three}, the sum of the two cap volumes is
\[
\text{Vol}(B_1\cap B_2)
=\frac{\pi}{3}\sum_{j=1}^2(r_j-h_j)^2(2r_j+h_j).
\]
Substituting \eqref{h_mu} and \eqref{h_r} and simplifying gives
\eqref{vol_3_balls_inter}.
\end{proof}

\section{Why one spherical shell is insufficient}
\label{three_dimensional_geometry}

In $\mathbb{R}^3$, a single spherical shell does not have the
property established by the planar example. The following lemma
shows that moving an arbitrarily small added ball from one side
of the origin to the other changes the central maximal function.
We use $A_W(x,r)=3\int_{B(xe_1,r)}W/(4\pi r^3)$.

Let $0<\delta<1$, let $B=B(0,1)$ and $B_\delta=B(0,\delta)$ in
$\mathbb{R}^3$, and put
\[
E_\delta=\{y\in\mathbb{R}^3:\delta<\|y\|<1\},
\qquad F_\delta=\chi_{E_\delta}.
\]

\begin{lemma}
\label{three_dimensional_shell_obstruction}
For every $0<\delta<1$, there is an $x_\delta>1$ such that
$r=x_\delta$ is the unique maximizing radius for
$A_{F_\delta}(x_\delta,r)$. One can take $x_\delta$ to be the larger
root of
\begin{equation}
\label{three_dimensional_shell_root}
4(1+\delta)x^2-8(1+\delta+\delta^2)x
       +3(1+\delta)(1+\delta^2)=0.
\end{equation}
In particular,
\[
x_{1/2}=\frac{14+\sqrt{61}}{12}.
\]
Moreover, for any $\epsilon>0$ and $0<\sigma<a$ with $a+\sigma<\delta$,
the two functions
\[
f_+=F_\delta+\epsilon\chi_{B(ae_1,\sigma)},
\qquad
f_-=F_\delta+\epsilon\chi_{B(-ae_1,\sigma)}
\]
have different maximal functions at $X=x_\delta e_1$.
\end{lemma}

\begin{proof}
For $x>1$, put $A(x,r)=A_{F_\delta}(x,r)$. The three-dimensional
intersection formula \eqref{vol_3_balls_inter} gives the following
derivatives in the three partial-overlap regimes:
\begin{equation}
\label{three_dimensional_shell_derivatives}
\frac{\partial A}{\partial r}(x,r)=
\begin{cases}
\displaystyle
\frac{3\bigl(r^2-(x-1)^2\bigr)
              \bigl(x^2+2x-3-r^2\bigr)}{16xr^4},
 &x-1<r<x-\delta,\\[8pt]
\displaystyle
\frac{3(1-\delta)}{16xr^4}
       \bigl(C_\delta(x)-2(1+\delta)r^2\bigr),
 &x-\delta<r<x+\delta,\\[8pt]
\displaystyle -\frac{3}{16xr^4}K_\delta(x,r),
 &x+\delta<r<x+1,
\end{cases}
\end{equation}
where
\begin{align*}
C_\delta(x)
 &=6(1+\delta)x^2-8(1+\delta+\delta^2)x
       +3(1+\delta)(1+\delta^2),\\
K_\delta(x,r)
 &=r^4-2x^2r^2+2r^2+x^4-6x^2+8x-3-16x\delta^3.
\end{align*}
The expressions for $\partial A/\partial r$ agree at their common
endpoints. Also $A(x,r)=0$ for $r\leq x-1$, whereas
$A(x,r)=(1-\delta^3)/r^3$ for $r\geq x+1$.

The polynomial on the left of \eqref{three_dimensional_shell_root}
has positive leading coefficient and takes the value
\[
3\delta^3-5\delta^2-\delta-1<0
\]
at $x=1$. It therefore has a unique root $x_\delta>1$.
At this value of $x$, we have
$C_\delta(x)=2(1+\delta)x^2$. Hence the middle expression in
\eqref{three_dimensional_shell_derivatives} is positive for $r<x$
and negative for $r>x$.

In the first regime, the sign of the derivative is the sign of the
strictly decreasing factor $x^2+2x-3-r^2$. Since the derivative is
positive at $r=x-\delta$, it is positive throughout that regime.
In the last regime,
\[
\frac{\partial K_\delta}{\partial r}(x,r)
   =4r(r^2-x^2+1)>0.
\]
The derivative of $A$ is negative at $r=x+\delta$, so it remains
negative up to $r=x+1$ and is negative thereafter. Thus $A(x,\cdot)$
is strictly increasing on $(x-1,x)$ and strictly decreasing on
$(x,\infty)$, proving the assertion about the maximizing radius.

Set $X=x_\delta e_1$ and $m=M_cF_\delta(X)$. The ball
$B(X,x_\delta)$ contains $B(ae_1,\sigma)$ entirely. Therefore
\[
M_cf_+(X)\geq m+\epsilon\frac{\sigma^3}{x_\delta^3}.
\]
In contrast, an averaging ball centered at $X$ does not meet
$B(-ae_1,\sigma)$ until its radius exceeds $x_\delta+a-\sigma$.
Its contribution from this added ball is consequently at most
$\epsilon\sigma^3/(x_\delta+a-\sigma)^3$. Since its contribution
from the shell is at most $m$, we obtain
\begin{equation}
\label{three_dimensional_moving_ball_obstruction}
\begin{aligned}
M_cf_-(X)
&\leq m+\epsilon\frac{\sigma^3}{(x_\delta+a-\sigma)^3}\\
&<m+\epsilon\frac{\sigma^3}{x_\delta^3}
\leq M_cf_+(X).
\end{aligned}
\end{equation}
This holds for every $\epsilon>0$, however small.
\end{proof}

Both added balls in the lemma can be placed in an arbitrarily
small neighborhood of the origin. Thus a single uniform shell does
not allow the translation-invariant family used in the planar
construction. This conclusion concerns that particular construction;
it does not rule out other backgrounds or other pairs of functions.

\section{The central maximal function of the unit ball in \texorpdfstring{$\mathbb{R}^3$}{R3}}

In dimension three, the central maximal function of the unit ball
has a particularly simple form.

\begin{theorem}
For $B=\{x\in\mathbb{R}^3:\|x\|<1\}$,
\begin{equation}
\label{maximal_ball}
M_c(\chi_B)(x)=
\begin{cases}
1,&\|x\|<1,\\[2pt]
\displaystyle\frac12,&\|x\|=1,\\[6pt]
\displaystyle\frac12-\frac12
 \sqrt{\frac{\|x\|-1}{\|x\|+3}}
 \left(1+\frac{2}{\|x\|}\right),&\|x\|>1.
\end{cases}
\end{equation}
For $\|x\|>1$ the unique maximizing radius is
\[
r_x=\sqrt{(\|x\|-1)(\|x\|+3)}.
\]
\end{theorem}

\begin{proof}
For $x>1$ and $x-1<r<x+1$, formula \eqref{vol_3_balls_inter} gives
\begin{equation}
\label{average}
A_{\chi_B}(x,r)
 =\frac{(r-x+1)^2
       \bigl(x^2+2x(r+1)-3(r-1)^2\bigr)}{16xr^3}.
\end{equation}
Its derivative is the first expression in
\eqref{three_dimensional_shell_derivatives}, throughout this whole
partial-overlap interval. It changes sign from positive to negative
at $r=\sqrt{x^2+2x-3}$, which lies strictly between $x-1$ and $x+1$.
Before that interval the average is zero; after it the average is
$r^{-3}$ and is decreasing. Substitution of the maximizing radius
into \eqref{average} proves \eqref{maximal_ball} for exterior centers.

For $x=1$ and $0<r<2$, \eqref{average} reduces to
\[
A_{\chi_B}(1,r)=\frac12-\frac{3r}{16}.
\]
For $r\geq2$ the average is $r^{-3}\leq1/8$. Thus the supremum
is $1/2$, approached as $r\downarrow0$. \end{proof}

Consequently,
\begin{equation}
\label{three_dimensional_boundary_asymptotic}
\frac12-M_c(\chi_B)(x)\sim\frac34(\|x\|-1)^{1/2}
       \qquad\text{as }\|x\|\downarrow1.
\end{equation}

\begin{remark}
The boundary asymptotic \eqref{three_dimensional_boundary_asymptotic}
is related to the level-set geometry of maximal functions; see
Hagelstein and Parissis~\cite{hagelsteinparissis2015}.
\end{remark}

\section*{Acknowledgement}

The author is grateful to Paul Hagelstein for his hospitality during
the author's visit to Baylor University in September 2023 and for
the opportunity to give a talk at a seminar. He also thanks him for
posing the question considered in this paper during their conversations.

The author also thanks Yuri Kryakin for the idea of shortening the
proof, suggested by a large language model.

The author has freely used results of mathematicians published in
the open literature, as well as various large language models.

\end{document}